\documentclass[11pt,reqno]{amsart}
\usepackage{enumerate}
\usepackage{amsfonts, amsmath, amssymb, amscd, amsthm, bm}
\usepackage{enumerate}
\usepackage{url}
\usepackage[linktocpage=true,colorlinks,citecolor=magenta,linkcolor=blue,urlcolor=magenta]{hyperref}
\usepackage{multicol}
\usepackage{comment}
\usepackage{amsrefs}

 \newtheorem{thm}{Theorem}[section]
 \newtheorem{cor}[thm]{Corollary}
  \newtheorem{con}[thm]{Conjecture}
 \newtheorem{lem}[thm]{Lemma}
 
 \newtheorem{prop}[thm]{Proposition}
 \theoremstyle{definition}
 
 \theoremstyle{remark}
 \newtheorem{rem}[thm]{Remark}
 \newtheorem*{ack}{Acknowledgments}

 \theoremstyle{claim}
 \numberwithin{equation}{section}
 \newcommand{\RR}{{\mathbb R}}

\begin{document}
\title[Volume preserving flow by $k$th mean curvature]
{Volume preserving flow by powers of $k$-th mean curvature}

\author[B. Andrews]{Ben Andrews}
\author[Y. Wei]{Yong Wei}
\address{Mathematical Sciences Institute,
Australia National University,
ACT 2601 Australia}
\email{\href{mailto:ben.andrews@anu.edu.au}{ben.andrews@anu.edu.au}, \href{mailto:yong.wei@anu.edu.au}{yong.wei@anu.edu.au}}

\date{\today}
\thanks {This research was supported by Laureate Fellowship FL150100126 of the Australian Research Council.}
\keywords {Mixed volume, convex body, curvature measure.}
\subjclass[2010]{53C44; 52A39}

%%% ----------------------------------------------------------------------

\begin{abstract}
We consider the flow of closed convex hypersurfaces in Euclidean space $\mathbb{R}^{n+1}$ with speed given by a power of the $k$-th mean curvature $E_k$ plus a global term chosen to impose a constraint involving the enclosed volume $V_{n+1}$ and the mixed volume $V_{n+1-k}$ of the evolving hypersurface.
We prove that if the initial hypersurface is strictly convex, then the solution of the flow exists for all time and converges to a round sphere smoothly.  No curvature pinching assumption is required on the initial hypersurface.
\end{abstract}

\maketitle
%\tableofcontents

\section{Introduction}
Let $X_0: M^n\to \mathbb{R}^{n+1}$ be a smooth embedding such that $M_0=X_0(M)=\partial\Omega_0$ is a closed strictly convex hypersurface in $\mathbb{R}^{n+1}$. We consider the smooth family of immersions $X:M^n\times [0,T)\rightarrow \mathbb{R}^{n+1}$ satisfying
\begin{equation}\label{flow-VMCF}
 \left\{\begin{aligned}
 \frac{\partial}{\partial t}X(x,t)=&~(\phi(t)-E_k^{{\alpha}/k}(x,t))\nu(x,t),\\
 X(\cdot,0)=&~X_0(\cdot),
  \end{aligned}\right.
 \end{equation}
where $\alpha>0$, $\nu(x,t)$ is the outward unit normal of the hypersurface $M_t=X(M,t)=\partial\Omega_t$, $k\in\{1,\cdots,n\}$ and $E_k$ is the $k$-th mean curvature of  $M_t$ defined as the normalized $k$-th elementary symmetric functions of the principal curvatures $(\kappa_1,\cdots,\kappa_n)$ of $M_t$:
 \begin{equation}\label{eq:defEk}
   E_{k}={\binom{n}{k}}^{-1}\sum_{1\leq i_1<\cdots<i_k\leq n}\kappa_{i_1}\cdots \kappa_{i_k}.
 \end{equation}
Clearly, $E_1=H/n$ and $E_n=K$ are the mean curvature and Gauss curvature of $M_t$ respectively. We also set $E_0=1$. The global term $\phi(t)$ in the flow \eqref{flow-VMCF} will be chosen to preserve a constraint involving the enclosed volume $V_{n+1}=(n+1)|\Omega_t|$ and the mixed volume $V_{n+1-k}(\Omega_t)$.  To describe this constraint precisely, we first briefly recall the mixed volumes of convex bodies:

Let $\Omega_1,\cdots,\Omega_{n+1}$ be convex bodies in $\mathbb{R}^{n+1}$. The mixed volumes are defined as the coefficients of the volume of the Minkowski sum $\sum_{i=1}^{n+1}\epsilon_i\Omega_i$:
\begin{equation*}%\label{Mix-V}
  V[\Omega_1,\cdots,\Omega_{n+1}]=(n+1)\frac{\partial^{n+1}}{\partial\epsilon_1\cdots\partial\epsilon_{n+1}}\mathrm{Vol}\left(\sum_{i=1}^{n+1}\epsilon_i\Omega_i\right).
\end{equation*}
In particular, the mixed volume (or quermassintegral) $V_{j}(\Omega)$ of a convex body $\Omega$ in $\mathbb{R}^{n+1}$ is defined as the following mixed volume of $\Omega$ with the unit ball $B$:
\begin{equation}\label{s1:MixedV}
  V_{j}(\Omega)=V[\underbrace{\Omega,\cdots,\Omega}_{j},\underbrace{B,\cdots,B}_{n+1-j}].
\end{equation}
In particular, $V_{n+1}(\Omega)=(n+1)\mathrm{Vol}(\Omega)$ and $V_0(\Omega)=(n+1)\mathrm{Vol}(B)=\omega_n$, where $\omega_n$ denotes the area of the unit sphere $\mathbb{S}^n$. If $\partial\Omega$ is $C^2$, the mixed volumes of $\Omega$ are related to the integral of $k$-th mean curvature over the boundary of $\Omega$ by the formula
\begin{equation}\label{def-quermass}
   V_{n-j}(\Omega)=\int_{\partial\Omega}E_{j}d\mu,\quad j=0,1,\cdots,n.
\end{equation}

For convenience we define the $j$-radius $r_j(\Omega)$ of $\Omega$ for $j=1,\cdots,n+1$ by
$$
r_j(\Omega) = \left(\frac{V_j(\Omega)}{V_j(B)}\right)^{\frac1j}=\omega_n^{-1/j}V_j(\Omega)^{1/j}.
$$
The $j$-radius is therefore the radius of the ball with the same value of $V_j$ as $\Omega$.  The Alexandrov-Fenchel inequalities \eqref{AF-k=n+1} imply that $r_j(\Omega)$ is non-increasing in $j$.

Now we consider a general smooth function $G:\ \{(a,b):\ a\geq b>0\}\to\mathbb{R}_+$ which is non-decreasing in each argument and has non-vanishing derivative at each point.  We choose the global term $\phi(t)$ in the flow \eqref{flow-VMCF} to keep the function $G(r_{n+1-k}(\Omega_t),r_{n+1}(\Omega_t))$ constant in $t$.  Explicitly, the choice of $\phi(t)$ is given by
\begin{align}\label{s1:phi-3}
  \phi(t) =& \frac{G_a\omega_n^{\frac1{n+1}}V_{n+1}^{\frac{n}{n+1}}\int_{M_t}E_k^{1+\frac{\alpha}{k}}
  +G_b\omega_n^{\frac{1}{n+1-k}}V_{n+1-k}^{\frac{n-k}{n+1-k}}\int_{M_t}E_k^{\frac{\alpha}{k}}}
  {G_a\omega_n^{\frac1{n+1}}V_{n+1}^{\frac{n}{n+1}}V_{n-k}
  +G_b\omega_n^{\frac{1}{n+1-k}}V_{n+1-k}^{\frac{n-k}{n+1-k}}V_n},
\end{align}
where $G_a$ and $G_b$ are the partial derivatives of $G$ with respect to the first and second variables respectively.

The main result that we prove in this paper is the following.
\begin{thm}\label{thm1-1}
Fix $k\in \{1,\cdots,n\}$ and $\alpha>0$.  Then for any smooth embedding $X_0: M^n\to \mathbb{R}^{n+1}$ such that $M_0=X_0(M)=\partial\Omega_0$ is a closed strictly convex hypersurface in $\mathbb{R}^{n+1}$, the flow \eqref{flow-VMCF} with global term $\phi(t)$ given by \eqref{s1:phi-3} has a smooth strictly convex solution $M_t=\partial\Omega_t$ for all time $t\in [0,\infty)$, and $M_t$ converges smoothly as $t\to\infty$ to a sphere $S^n(\bar{r})$ of radius $\bar{r}$ determined by $G(\bar r,\bar r) = G(r_{n+1-k}(\Omega_0),r_{n+1}(\Omega_0))$.
\end{thm}

\begin{rem}
The choices $G(a,b)=b$ and $G(a,b)=a$ correspond to the flows which preserve the enclosed volume $V_{n+1}(\Omega_t)$ or the mixed volume $V_{n+1-k}(\Omega_t)$ respectively, so Theorem \ref{thm1-1} contains the result for these flows as special cases.
\end{rem}

Certain cases of this result have been proved previously, as well as several other related results:  The first case treated was the volume-preserving mean curvature flow ($k=1$, $\alpha=1$ and $G(a,b)=b$) which was considered by Huisken \cite{huis-87}.  The crucial estimate was a curvature ratio bound, and the argument was adapted from that used previously for the mean curvature flow \cite{huis-1984}.
Similar methods apply to the area-preserving mean curvature flow ($k=1$, $\alpha=1$ and $G(a,b)=a$), as shown by McCoy \cite{McC2003}.  In fact the argument using pointwise curvature estimates holds very generally for flows with $\alpha=1$, allowing the treatment of mean curvature flows preserving other mixed volumes \cite{McC2004} and also flows by homogeneous degree one curvature functions in a large class \cite{McC2005,McC2017} with a constraint on any of the mixed volumes $V_j(\Omega_t)$, despite the fact that there is no variational structure and no monotone isoperimetric ratios known for such flows.

Several works have treated flows with $\alpha>1$ by following the same argument using curvature ratio bounds:
For contraction flows (with $\phi(t)=0$) it was observed that for $\alpha>1$, sufficiently strong curvature ratio bounds are preserved \cite{Chow85,Schu06,Aless-Sin,And-McC2012}.  This is also true for the constrained flows with $\alpha>1$, allowing such flows to be understood provided one can overcome the degeneracy which arises when the speed becomes small:  If one can show that solutions remain smooth, with uniform estimates as time approaches infinity, then the curvature pinching estimate implies that the limit must be a sphere.  Several techniques have been used to handle this degeneracy:  For constrained flows by powers of $E_k$, Cabezas-Rivas and Sinestrari \cite{Cab-Sin2010} observed that (once the curvature ratio bound is established) the equation equation for $E_k$ has the structure of a porous medium equation, and in particular estimates for porous medium equations \cite{DiB-F} imply that $E_k$ is H\"older continuous (estimates of this kind had been employed previously for contraction flows by powers of mean curvature by Schulze \cite{Schu06} and for powers of scalar curvature by Alesssandroni and Sinestrari \cite{Aless-Sin}).  This allows the proof of convergence to a sphere to be completed.  McCoy \cite[Section 6]{McC2017} also showed that with sufficiently strong curvature ratio bounds on the initial hypersurface, the convergence result can be established for $\alpha>1$ for a much wider class of speeds, by using spherical barriers and an adaptation of an estimate of Smoczyk \cite{Smo98} to derive a lower speed bound (here the argument uses the fact that hypersurfaces with a strong curvature ratio bound are close to spheres, in order to make the spherical barriers effective).  This removes the need for porous medium estimates, but requires stronger curvature pinching assumptions.

It seems to be true, however, that curvature pinching estimates are less decisive for constrained flows than they are for the corresponding contraction flows:  The first author treated anisotropic analogues of the volume-preserving mean curvature flow \cite{And2001-Aniso}, and proved that solutions converge to the Wulff shape corresponding to the anisotropy, despite the fact that no curvature pinching estimate could be obtained.  Instead, the convergence argument was based on an improving isoperimetric ratio (and the fact that a bound on isoperimetric ratio for a convex hypersurface implies bounds on the ratio of diameter to inradius).   The improving isoperimetric ratio was also used by Sinestrari \cite{Sin-2015} to prove the convergence of the volume-preserving or area-preserving flows by powers of mean curvature (this corresponds to the case $k=1$ of Theorem \ref{thm1-1} with $G(a,b)=b$ or $G(a,b)=a$).  In that paper the result holds for arbitrary powers $\alpha>0$, and no initial pinching condition is required.  Instead the isoperimetric bounds are used to deduce bounds above on diameter and below on inradius, and the porous medium estimates of \cite{DiB-F} are applied to give H\"older continuity of the mean curvature.  From this it is possible to deduce that the solution remains regular and converges to a smooth limit which has constant mean curvature and is therefore a sphere.

More recently, Bertini and Sinestrari \cite{Bert-Sin} have considered flows by very general non-homogeneous increasing functions of the mean curvature, with a constraint on the enclosed volume.  Again, the isoperimetric ratio bound plays a crucial role in controlling the geometry, but the authors also derive a lower bound on the speed directly from the maximum principle, making the remaining analysis much easier and in particular removing the need to use porous medium estimates.

Our argument to prove Theorem \ref{thm1-1} exploits the isoperimetric bounds to control the geometry of the evolving hypersurfaces:  For any $k\in\{1,\cdots,n\}$ and any $\alpha>0$, we show that $V_{n+1-k}(\Omega_t)$ is non-increasing, while $V_{n+1}(\Omega_t)$ is non-decreasing.  This implies a time-independent bound on diameter and a time-independent lower bound on inradius, and these allow us to derive an upper bound on $E_k$ using the method of Tso \cite{Tso85}.  This is sufficient for us to deduce that the solution exists and remains smooth and strictly convex for all positive times.
However, for $k>1$ new ideas are needed to prove convergence to a sphere:  Without a curvature ratio bound, the flow can no longer be written as a porous medium equation with uniformly elliptic coefficients, and so this route to regularity cannot be used.  The lower speed bound of Bertini and Sinestrari is also not available without some kind of curvature pinching control.  Instead, we use some machinery from the theory of convex bodies:   In particular, we use the Blaschke selection theorem to show that the enclosed regions $\Omega_t$ of $M_t$ converge in Hausdorff distance (for a subsequence of times) to a limiting convex region $\hat\Omega$ as $t\to\infty$.  We then deduce from the evolution of $V_{n+1-k}$ that $\hat\Omega$ satisfies
\begin{equation*}
  \mathcal{C}_{n-k}(\hat{\Omega},\beta)=~c~\mathcal{C}_n(\hat{\Omega},\beta)
\end{equation*}
for any Borel set $\beta$ in $\mathbb{R}^{n+1}$, where ${\mathcal C}_{n-k}$ and ${\mathcal C}_n$ are the curvature measures of $\hat\Omega$, and $c$ is a constant.  A theorem of Schneider \cite{Schn79} (a generalization of the classical Alexandrov Theorem) can be used to conclude that $\hat{\Omega}$ is a ball. Using the monotonicity of the isoperimetric ratio again, the Hausdorff convergence of the whole family $\Omega_t$ to a ball follows easily.  With the help of the Hausdorff convergence, we can adapt an idea of Smoczyk \cite[Proposition 4]{Smo98} (see also \cite{Andrews-McCoy-Zheng,McC2017}) to prove a uniform positive lower bound on $E_k$. Then the smooth convergence of the flow follows by a standard argument.

In the end of this paper, we discuss several generalisations of Theorem \ref{thm1-1}:

First, we consider generalisations in which the driving speed is a non-homogeneous function $\mu(E_k^{\frac1k})$.  The flows is question then have the form
\begin{equation}\label{eq:flow-nonhom}
 \left\{\begin{aligned}
 \frac{\partial}{\partial t}X(x,t)=&\left(\phi(t)-\mu\left(E_k^{1/k}(x,t)\right)\right)\nu(x,t),\\
 X(\cdot,0)=&~X_0(\cdot),
  \end{aligned}\right.
 \end{equation}
 where $\mu$ is smooth and positive with positive derivative, and satisfies some structural assumptions near zero and near infinity.  Such flows were treated in the case $k=1$ by Bertini and Sinestrari \cite{Bert-Sin}, and the methods of this paper allow us to produce a
similarly general result for all $k$.

  Second, we generalise the constrained flows considered previously by considering flows in which the enclosed volume is monotone:  That is, in the flow \eqref{flow-VMCF} (or non-homogeneous generalisations of the form \eqref{eq:flow-nonhom}) we require the global term $\phi(t)$ to be smooth and satisfy
\begin{equation*}%\label{s1:phi-0}
  \phi(t)\geq \frac 1{V_n(\Omega_t)}\int_{M_t}\mu(E_k^{1/k})d\mu_t.
\end{equation*}
Under this assumption (and some further asymptotic assumptions if $\mu$ is not homogeneous) we can also show that the flow \eqref{flow-VMCF} has a smooth strictly convex solution $M_t$ for all time $t\in [0,\infty)$, and either
\begin{itemize}
  \item[(i)] the volume of $\Omega_t$ is uniformly bounded above and $M_t$ converges smoothly as $t\to\infty$ to a sphere $S^n_{\bar{r}}(p)$; or
  \item[(ii)] the volume of $\Omega_t$ goes to infinity and $M_t$ is asymptotic to an expanding sphere with radius depending on $\phi(t)$.
\end{itemize}
The monotonicity of an isoperimetric ratio is a key ingredient in proving the above result.

Third, we briefly discuss anisotropic generalisations of the flows.  We show that anisotropic analogues of the results of Bertini and Sinestrari \cite{Bert-Sin} hold (corresponding to $k=1$), and also provide correspondingly strong results for $k=n$ by making use of some stability results for inequalities between mixed volumes.  We discuss the corresponding problem for $1<k<n$, and identify an natural conjecture concerning hypersurfaces satisfying relations between the corresponding anisotropic curvature measures which would allow the more general anisotropic results to be proved.

Finally, we observe that the results for functions of mean curvature can be generalised to volume-preserving flows involving much larger classes of flows involving uniformly monotone functions of principal curvatures.  In these cases we no longer have a monotone isoperimetric inequality, but we can deduce diameter bounds from an Alexandrov reflection argument, and then derive inradius lower bounds from the preservation of enclosed volume.  In these cases the lower speed bound of Bertini and Sinestrari \cite{Bert-Sin} applies, so we can deduce smooth convergence to a limiting hypersurface along a subsequence of times approaching infinity.  The convergence to a sphere then follows from a strong maximum principle applied in the Alexandrov reflection argument.

See \S \ref{sec:8} for the detailed discussion of these generalisations.

\begin{ack}
The authors would like to thank James McCoy for discussions.
\end{ack}

\section{Preliminaries}
In this section, we collect some evolution equations along the flow \eqref{flow-VMCF} and preliminary results on convex bodies.
\subsection{Evolution equations}
Along the flow \eqref{flow-VMCF}, we have the following evolution equations on the induced metric $g_{ij}$, unit outward normal $\nu$, the induced area element $d\mu_t$ and $l$-th mean curvature of $M_t$:
\begin{align}%\label{}
  \frac{\partial}{\partial t}g_{ij} =&~ 2(\phi(t)-E_k^{{\alpha}/k})h_{ij} \label{evl-g}\displaybreak[0]\\
  \frac{\partial}{\partial t}\nu=&~\nabla E_k^{\alpha/k}\label{evl-nu}\displaybreak[0]\\
   \frac{\partial}{\partial t}d\mu_t=& ~nE_1 (\phi(t)-E_k^{{\alpha}/k})d\mu_t\label{evl-dmu}\displaybreak[0]\\
    \frac{\partial}{\partial t}E_l=&~\frac{\partial E_l}{\partial h_{ij}}\nabla_j\nabla_iE_k^{{\alpha}/k}-(\phi(t)-E_k^{{\alpha}/k})(nE_1E_l-(n-l)E_{l+1}), \label{evl-El}
\end{align}
where $l=1,\cdots,n-1$, and $\nabla$ denotes the Levi-Civita connection with respect to the induced metric $g_{ij}$ on $M_t$. The proof is by similar calculations as in \cite{huis-1984}.

\subsection{Gauss map parametrisation of the flow}\label{sec:2-3}
The convex hypersurfaces can be parametrised via the Gauss map. Given a smooth strictly convex hypersurface $M$ in $\mathbb{R}^{n+1}$, the support function $s: \mathbb{S}^n\to\mathbb{R}$ of $M$ is defined by $s(z)=\sup\{\langle x,z\rangle:x\in\Omega\}$, where $\Omega$ is the convex body enclosed by $M$. Then the hypersurface $M$ is given by the embedding
\begin{equation*}
  X(z)=s(z)z+\bar{\nabla}s(z),
\end{equation*}
where $\bar{\nabla}$ is the gradient with respect to the round metric $\bar{g}_{ij}$ on $\mathbb{S}^n$. The principal radii of curvature, or the inverses of the principal curvatures, are the eigenvalues of
\begin{equation*}
  \tau_{ij}=\bar{\nabla}_i\bar{\nabla}_js+\bar{g}_{ij}s
\end{equation*}
 with respect to $\bar{g}_{ij}$.

Denote $F=E_k^{1/k}$ and define the function $F_*$ by
\begin{equation*}
  F_*(x_1,\cdots,x_n)=F(x_1^{-1},\cdots,x_n^{-1})^{-1},
\end{equation*}
which is concave in its argument. The solution of the flow \eqref{flow-VMCF} is then given, up to a time-dependent diffeomorphism,  by solving the scalar parabolic equation on $\mathbb{S}^n$
 \begin{equation}\label{flow-gauss}
   \frac{\partial s}{\partial t}=\Psi(\tau_{ij})+\phi(t)=-F_*(\tau_{ij})^{-\alpha}+\phi(t)
 \end{equation}
for the support function $s$, where $F_*(\tau_{ij})$ can be viewed as the function $F_*$ evaluated at the eigenvalues of $\tau_{ij}$. This implies the existence of a smooth solution of \eqref{flow-VMCF}
for a short time for any smooth, strictly convex initial hypersurface.
\begin{lem}[cf. Lemma 10 in \cite{Andrews-McCoy-Zheng}]
Under the flow \eqref{flow-gauss}, the speed $\Psi$ evolves according to
\begin{equation}\label{flow-Psi}
  \frac{\partial\Psi}{\partial t}=\dot{\Psi}^{kl}\bar{\nabla}_k\bar{\nabla}_l\Psi+(\Psi+\phi(t))\dot{\Psi}^{kl}\bar{g}_{kl}.
\end{equation}
The inverse second fundamental form $\tau_{ij}$ evolves by
\begin{align*}%\label{flow-tau1}
  \frac{\partial}{\partial t}\tau_{ij}=&\dot{\Psi}^{kl}\bar{\nabla}_k\bar{\nabla}_l\tau_{ij}+\ddot{\Psi}^{kl,pq}\bar{\nabla}_i\tau_{kl}\bar{\nabla}_j\tau_{pq}\nonumber\\
  &\quad -\dot{\Psi}^{kl}\bar{g}_{kl}\tau_{ij}+((1-\alpha)\Psi+\phi(t))\bar{g}_{ij},
\end{align*}
and equivalently
\begin{align}\label{flow-tau2}
  \frac{\partial}{\partial t}\tau_{ij}=&\dot{\Psi}^{kl}\bar{\nabla}_k\bar{\nabla}_l\tau_{ij}+\alpha F_*^{-\alpha-1}\ddot{F}_*^{kl,pq}\bar{\nabla}_i\tau_{kl}\bar{\nabla}_j\tau_{pq}-\alpha(\alpha+1)F_*^{-\alpha-2}\bar{\nabla}_iF_*\bar{\nabla}_jF_*\nonumber\\
  &\quad -\alpha F_*^{-\alpha-1}\dot{F}_*^{kl}\bar{g}_{kl}\tau_{ij}+((\alpha-1)F_*^{-\alpha}+\phi(t))\bar{g}_{ij},
\end{align}
\end{lem}

\subsection{Mixed volumes and Alexandrov-Fenchel inequalities}\label{sec:mixed vol}

For any convex body $\Omega$ in $\mathbb{R}^{n+1}$, we have defined the mixed volumes $V_{n+1-j}(\Omega)$ in \eqref{s1:MixedV}. The Alexandrov-Fenchel inequalities (see Equation (7.66) of \cite{Schn}) state that
\begin{equation}\label{AF-00}
  V_{n+1-j}^{k-i}(\Omega)~\geq~V_{n+1-i}^{k-j}(\Omega)V_{n+1-k}^{j-i}(\Omega)
\end{equation}
for any convex body $\Omega$ and all $0\leq i<j<k\leq n+1$.  A special case of \eqref{AF-00} is the following
\begin{equation}\label{AF-1}
  V_{n+1-j}^2(\Omega)~\geq~V_{n-j}(\Omega)V_{n+2-j}(\Omega),\quad j=1,\cdots,n
\end{equation}
by letting $i=j-1$ and $k=i+1$ in \eqref{AF-00}. When $j=n$, the equality holds in \eqref{AF-1} if and only if $\Omega$ is homothetic to a ball. In fact, we have a stability result for \eqref{AF-1} when $j=n$:

\begin{lem}[\cite{Schn}*{Theorem 7.6.6 and Lemma 7.6.4}]\label{s2:stab-1}
\begin{equation}\label{AF-2}
  V_{1}^2(\Omega)-V_{0}(\Omega)V_{2}(\Omega)~\geq~\frac{C(n)}{V_1(\Omega)^n}\mathrm{d}_{H}(\Omega,B_{\Omega})^{n+2},
\end{equation}
where $\mathrm{d}_H$ is the Hausdorff distance of two subsets in $\mathbb{R}^{n+1}$, and $B_{\Omega}$ is a ball with the same Steiner point and mean width as $\Omega$.
\end{lem}

Lemma \ref{s2:stab-1} will be used to prove the Hausdorff convergence of $M_t$ to a sphere as $t\to\infty$ in the case $k=n$ of the flow \eqref{flow-VMCF}. Another special case of \eqref{AF-00} is when $k=n+1$, it reduces to
\begin{equation}\label{AF-k=n+1}
  V_{n+1-j}^{n+1-i}(\Omega)~\geq~\omega_n^{j-i}V_{n+1-i}^{n+1-j}(\Omega)
\end{equation}
for all $0\leq i<j<n+1$. In this case, the equality of \eqref{AF-k=n+1} also characterizes balls.

We also need the following useful lemma concerning the continuity of the mixed volumes.
\begin{lem}\label{s2:lem2}
If $0<R_1\leq R_2<\infty$, there exists a constant $C=C(R_1,R_2)$ such that any two convex bodies $\Omega_1,\Omega_2$ which can be translated to have $R_1B\subset \Omega_1,\Omega_2\subset R_2B$, where $B$ is the unit Euclidean ball, satisfy
\begin{equation*}%\label{s6:lem2-1}
  \left|\frac{V_{n-k}(\Omega_1)}{V_n(\Omega_1)}-\frac{V_{n-k}(\Omega_2)}{V_n(\Omega_2)}\right|~\leq ~C d_{\mathrm{H}}(\Omega_1,\Omega_2)
\end{equation*}
\end{lem}
\proof
Recall that the Hausdorff distance between two regions is defined as
\begin{equation*}
  d_\mathrm{H}(\Omega_1,\Omega_2)=~\inf\left\{\delta: \Omega_1\subset \Omega_2+\delta B,~\mathrm{and}~\Omega_2\subset \Omega_1+\delta B\right\}
\end{equation*}
Let $\delta=d_\mathrm{H}(\Omega_1,\Omega_2)$. Then $\delta<R_2$ and $\Omega_1\subset \Omega_2+\delta B$.
\begin{align*}
  V_{n-k}(\Omega_1)-V_{n-k}(\Omega_2) =&~V[\underbrace{\Omega_1,\cdots,\Omega_1}_{n-k},\underbrace{B,\cdots,B}_{k+1}]- V[\underbrace{\Omega_2,\cdots,\Omega_2}_{n-k},\underbrace{B,\cdots,B}_{k+1}] \\
  \leq &~ V[\underbrace{\Omega_2+\delta B,\cdots,\Omega_2+\delta B}_{n-k},\underbrace{B,\cdots,B}_{k+1}]- V[\underbrace{\Omega_2,\cdots,\Omega_2}_{n-k},\underbrace{B,\cdots,B}_{k+1}]\\
  \leq&~ C\delta R_2^{n-k-1}\omega_n
\end{align*}
Similarly,
\begin{align*}
  V_{n}(\Omega_1)-V_{n}(\Omega_2)\leq&~ C\delta R_2^{n-1}\omega_n.
\end{align*}
We also have
\begin{equation*}
  R_1^{n-k}\omega_n~\leq V_{n-k}(\Omega_1), V_{n-k}(\Omega_2)~\leq~R_2^{n-k}\omega_n.
\end{equation*}
Hence,
\begin{align*}
  \left|\frac{V_{n-k}(\Omega_1)}{V_n(\Omega_1)}-\frac{V_{n-k}(\Omega_2)}{V_n(\Omega_2)}\right|\leq & ~V_n(\Omega_1)^{-1}\left|V_{n-k}(\Omega_1)-V_{n-k}(\Omega_2)\right| \\
   & \quad +\frac{V_{n-k}(\Omega_2)}{V_n(\Omega_1)V_n(\Omega_2)}\left|V_n(\Omega_1)-V_n(\Omega_2)\right|\\
   \leq&~C\delta\left(R_2^{n-k-1}R_1^{-n}+R_2^{2n-k-1}R_1^{-2n}\right).
\end{align*}
This completes the proof.
\endproof

\subsection{Curvature measures}\label{sec:2-4}
To characterize the limit of the flow, we will employ the curvature measure of convex bodies. Given a convex body $\Omega$ in $\mathbb{R}^{n+1}$, $\rho>0$ and any Borel set $\beta\in \mathcal{B}(\mathbb{R}^{n+1})$, we consider the following local parallel set
 \begin{equation*}
   A_{\rho}(\Omega,\beta):=\{x\in \mathbb{R}^{n+1}:~0<d(\Omega,x)\leq\rho,~p(\Omega,x)\in\beta\}
 \end{equation*}
which is the set of all points $x\in \mathbb{R}^{n+1}$ such that the distance $d(\Omega,x)\leq \rho$ and the nearest point $p(\Omega,x)$ belongs to $\beta$. By the theory of convex geometry (see \cite[\S 4.2]{Schn}), the area of $A_\rho(\Omega,\beta)$ is a polynomial in the parameter $\rho$: Precisely,
 \begin{align*}
   \mathcal{H}^{n+1}( A_{\rho}(\Omega,\beta))=&\frac 1{n+1}\sum_{m=0}^n\rho^{n+1-m}{\binom{n+1}{m}}\mathcal{C}_m(\Omega,\beta)
 \end{align*}
for $\beta\in \mathcal{B}(\mathbb{R}^{n+1})$ and $\rho>0$. The coefficients $\mathcal{C}_0(\Omega,\cdot), \cdots, \mathcal{C}_n(\Omega,\cdot)$ are called the \emph{curvature measures} of the convex body $\Omega$, which are Borel measures on $\mathbb{R}^{n+1}$. If $\Omega$ is $(n+1)$-dimensional, then (see Theorem 4.2.3 of \cite{Schn})
\begin{equation}\label{s2:cur-meas-n}
  \mathcal{C}_n(\Omega,\beta)=~\mathcal{H}^n(\beta\cap\partial\Omega)
\end{equation}
for $\beta\in \mathcal{B}(\mathbb{R}^{n+1})$. The function $\Omega\mapsto \mathcal{C}_{m}(\Omega,\cdot)$ is weakly continuous with respect to Hausdorff distance (equivalent to the statement that $\int f d{\mathcal C}_m(\Omega)=\lim_{i\to\infty}\int f d{\mathcal C}_m(\Omega_i)$ whenever $f$ is a bounded continuous function on $\RR^{n+1}$ and $d_{\mathcal H}(\Omega_i,\Omega)\to 0$). If $\Omega$ is a convex body of class $C^2_+$, and the boundary $\partial \Omega$ has principal curvatures $\kappa=(\kappa_1,\cdots, \kappa_n)$, then the curvature measure has the equivalent form
\begin{align}
 \mathcal{ C}_{m}(\Omega,\beta)=&\int_{\beta\cap \partial \Omega}E_{n-m}(\kappa)d\mathcal{H}^n\label{s2:cur-meas}
\end{align}
for any Borel set $\beta\in \mathcal{B}(\mathbb{R}^{n+1})$.

The following is a generalization of the classical Alexandrov Theorem in differential geometry.
\begin{thm}[Theorem 8.5.7 in \cite{Schn}]\label{s2:thm-meas}
Let $m\in\{0,\cdots,n-1\}$. If $\Omega$ is a convex body with nonempty interior, satisfying
\begin{equation*}
  \mathcal{C}_{m}(\Omega,\beta)=~c~ \mathcal{C}_{n}(\Omega,\beta)
\end{equation*}
for any Borel set $\beta\in \mathcal{B}(\mathbb{R}^{n+1})$, where $c>0$ is a constant, then $\Omega$ is a ball.
\end{thm}

\section{Monotonicity of the isoperimetric ratio}
For any convex body $\Omega$ in $\mathbb{R}^{n+1}$ and any integer $1\leq l\leq n$, define the following isoperimetric ratio:\begin{equation*}
{\mathcal I}_\ell(\Omega) = \frac{V_\ell(\Omega)^{n+1}}{V_{n+1}(\Omega)^\ell V_0^{n+1-\ell}}.
\end{equation*}
By the Alexandrov-Fenchel inequality \eqref{AF-k=n+1}, we have
\begin{equation}\label{s3:Il-2}
{\mathcal I}_\ell(\Omega)\geq 1,\quad 1\leq \ell\leq n,
\end{equation}
with equality if and only if $\Omega$ is a ball.

We note that all of these isoperimetric ratios are comparable, in the sense that a bound on any of these implies bounds on all of the others:  The fact that $r_m(\Omega)$ is non-increasing in $m$ implies that ${\mathcal I}_m(\Omega)^{\frac1{m(n+1)}}$ is non-increasing in $m$.  In the other direction, the Alexandrov-Fenchel inequality $V_i^{n+1-j}\geq V_{n+1}^{i-j}V_j^{n+1-i}$ for $i>j$ implies that ${\mathcal I}_m(\Omega)^{\frac{1}{n+1-m}}$ is non-decreasing in $m$.

\begin{prop}\label{prop-monot}
Let $M_t=X(M,t)=\partial\Omega_t$ be a smooth convex solution of the flow \eqref{flow-VMCF} on $[0,T)$ with the global term $\phi(t)$ given by \eqref{s1:phi-3}. Then $V_{n+1}(\Omega_t)$ is non-decreasing in $t$, $V_{n+1-k}(\Omega_t)$ is non-increasing in $t$, and the isoperimetric ratio ${\mathcal I}_{n+1-k}(\Omega_t)$ is monotone non-increasing in $t$.
\end{prop}

\proof
Since $\Omega_t$ is smooth and convex, we can write the mixed volumes $V_{n+1-k}(\Omega_t)$, $V_{n-k}(\Omega_t)$, and $V_n(\Omega_t)$ as in \eqref{def-quermass}.  Along the flow \eqref{flow-VMCF}, $V_{n+1}(\Omega_t)$ evolves by
\begin{align}\label{s3:Vn1}
  \frac d{dt}V_{n+1}(\Omega_t)=&(n+1)\frac d{dt}|\Omega_t|=(n+1)\int_{M_t}(\phi(t)-E_k^{{\alpha}/k})d\mu_t\nonumber\\
  =&(n+1)\left(\phi(t)V_n(\Omega_t)-\int_{M_t}E_k^{\alpha/k}d\mu_t\right).
\end{align}
Combining \eqref{evl-dmu}--\eqref{evl-El}, we also have
\begin{align}\label{s3:Vn1k}
  \frac d{dt}V_{n+1-k}(\Omega_t)=&~(n-k+1)\int_{M_t}E_{k}(\phi(t)-E_k^{{\alpha}/k})d\mu_t\nonumber\\
  =&(n+1-k)\left(\phi(t)V_{n-k}(\Omega_t)-\int_{M_t}E_k^{\frac{\alpha}k+1}d\mu_t\right).
\end{align}

The global term $\phi(t)$ is determined by the requirement that the flow \eqref{flow-VMCF} preserves the function $G(r_{n+1-k}(\Omega_t),r_{n+1}(\Omega_t))$, i.e.,
\begin{align}\label{s3:G1}
  0&=\frac d{dt}G\left(r_{n+1-k}(\Omega_t),r_{n+1}(\Omega_t)\right)\nonumber\\
  &=\frac{d}{dt}G\left(\omega_n^{-\frac{1}{n+1-k}}V_{n+1-k}(\Omega_t)^{\frac 1{n+1-k}},\omega_n^{-\frac{1}{n+1}}V_{n+1}(\Omega_t)^{\frac 1{n+1}}\right)\nonumber\\
  &=G_a\omega_n^{-\frac{1}{n+1-k}}V_{n+1-k}(\Omega_t)^{\frac {k-n}{n+1-k}}\frac 1{n+1-k}\frac d{dt}V_{n+1-k}(\Omega_t)\nonumber\\
  &\qquad +G_b\omega_n^{-\frac{1}{n+1}}V_{n+1}(\Omega_t)^{-\frac n{n+1}}\frac 1{n+1}\frac d{dt}V_{n+1}(\Omega_t)\nonumber\\
  &=\phi\left(G_a\omega_n^{-\frac{1}{n+1-k}}V_{n+1-k}^{\frac {k-n}{n+1-k}}V_{n-k}+G_b\omega_n^{-\frac{1}{n+1}}V_{n+1}(\Omega_t)^{-\frac n{n+1}}V_n
  \right)\nonumber\\&\quad\null-\left(G_a\omega_n^{-\frac{1}{n+1-k}}V_{n+1-k}^{\frac {k-n}{n+1-k}}\int_{M_t}E_k^{1+\frac{\alpha}{k}}
  +G_b\omega_n^{-\frac{1}{n+1}}V_{n+1}(\Omega_t)^{-\frac n{n+1}}\int_{M_t}E_k^\frac{\alpha}{k}
  \right).
\end{align}
Rearranging this gives the expression \eqref{s1:phi-3}.
Since $G_a\geq 0$ and $G_b\geq 0$, from the expression \eqref{s1:phi-3} and Jensen's inequality
\begin{equation}\label{s3:Jensen}
  \int_{M_t}E_k^{1+\frac{\alpha}{k}}d\mu_t~\geq ~\frac 1{|M_t|}\int_{M_t}E_{k}d\mu_t\int_{M_t}E_k^{\alpha/k}d\mu_t
\end{equation}
we always have
\begin{equation}\label{s3:phi-1}
  \frac 1{V_n(\Omega_t)}\int_{M_t}E_k^{\alpha/k} \leq \phi(t)\leq \frac{1}{V_{n-k}(\Omega_t)}\int_{M_t}E_k^{1+\frac{\alpha}{k}}.
\end{equation}
From the first inequality of \eqref{s3:phi-1} and equation \eqref{s3:Vn1} we deduce that $\frac{d}{dt}V_{n+1}(\Omega_t)\geq 0$, while from the second inequality of \eqref{s3:phi-1} and equation \eqref{s3:Vn1k} we deduce that $\frac{d}{dt}V_{n+1-k}(\Omega_t)\leq 0$.  The monotonicity of the isoperimetric ratio follows.
\endproof

\begin{prop}\label{s3:prop-2}
Let $M_t$ be a smooth convex solution of the flow \eqref{flow-VMCF} on $[0,T)$ with the global term $\phi(t)$ given by \eqref{s1:phi-3}. Then there exist constants $c_1,c_2, c_3,R_1,R_2$ depend only on $n,k,M_0$ such that
\begin{equation}\label{s3:AreaVol}%\label{prop2-2-ineq}
  0<c_1\leq |M_t|\leq c_2,\quad |\Omega_0|\leq |\Omega_t|\leq c_3
\end{equation}
and
\begin{equation}\label{s3:radius}
  0<R_1\leq \rho_-(t)\leq \rho_+(t)\leq R_2,
\end{equation}
where $\rho_+(t)=\rho_+(\Omega_t), \rho_-(t)=\rho_-(\Omega_t)$ are the inner radius and outer radius of $\Omega_t$ respectively.
\end{prop}
\proof
(i). Firstly, since the volume of $\Omega_t$ is non-decreasing along the flow \eqref{flow-VMCF}, we have $|\Omega_t|\geq |\Omega_0|$.  The isoperimetric inequality \eqref{s3:Il-2} then implies that
\begin{equation}\label{Mt-ld}
  |M_t|^{\frac 1n}\geq \omega_n^{\frac 1{n(n+1)}}|\Omega_t|^{\frac 1{n+1}}\geq \omega_n^{\frac 1{n(n+1)}}|\Omega_0|^{\frac 1{n+1}},
\end{equation}
which gives the lower bound of $|M_t|$. On the other hand, since $V_{n+1-k}(\Omega_t)$ is non-increasing along the flow, by the Alexandrov-Fenchel inequality \eqref{AF-k=n+1} we have
\begin{align}\label{Mt-ub}
V_{n+1-k}(\Omega_0)~\geq &~V_{n+1-k}(\Omega_t)\geq \omega_n^{\frac{k-1}n}V_n(\Omega_t)^{\frac{n+1-k}n}~=\omega_n^{\frac{k-1}n}|M_t|^{\frac{n+1-k}n}.
\end{align}
This gives an upper bound of $|M_t|$ by a constant depending only on $n,k,M_0$.  This in turn gives an upper bound of $|\Omega_t|$ with the help of \eqref{Mt-ld}.

(ii). From the Alexandrov-Fenchel inequality \eqref{AF-k=n+1} and Proposition \ref{prop-monot}, we can estimate the isoperimetric ratio ${\mathcal I}_n(\Omega_t)$:
\begin{equation}\label{I0-ub}
\sigma(t):={\mathcal I}_n(\Omega_t)^{\frac{1}{n(n+1)}}=\frac{r_n(\Omega_t)}{r_{n+1}(\Omega_t)} \leq \frac{r_{n+1-k}(\Omega_t)}{r_{n+1}(\Omega_t)}=
{\mathcal I}_{n+1-k}(\Omega_t)^{\frac{1}{(n+1)(n+1-k)}}.
\end{equation}
The right-hand side is bounded by the initial value ${\mathcal I}_{n+1-k}(\Omega_0)^{\frac{1}{(n+1)(n+1-k)}}$ (which we denote by $\sigma_0$) for any $t\geq 0$.
The Diskant inequality \cite{Schn}*{Equation (7.28)} then gives a lower bound on the inradius:
\begin{align*}
\rho_-(t)&\geq r_n(\Omega_t)-(r_n(\Omega_t)^{n+1}-r_{n+1}(\Omega_t)^{n+1})^{\frac1{n+1}}\\
&= r_{n+1}(\Omega_t)\left(\sigma(t)-\left(\sigma(t)^{n+1}-1\right)^\frac1{n+1}\right)\\
&\geq r_{n+1}(\Omega_0)\left(\sigma_0-\left(\sigma_0^{n+1}-1\right)^{\frac1{n+1}}\right).
\end{align*}
Note that the diameter is also controlled, since
$$
\frac{r_1(\Omega_t)}{r_{n+1}(\Omega_t)}={\mathcal I}_1(\Omega_t)^{\frac1{n+1}}\leq{\mathcal I}_{n+1-k}(\Omega_t)^{\frac{n}{k(n+1)}}\leq \sigma_0^{\frac{n(n+1-k)}{k}}.
$$
The $1$-radius $r_1(\Omega_t)$ controls the diameter, since $V_1(\Omega_t)\geq V_1(L) = c(n)D(\Omega_t)$ where $L$ is a line segment of maximal length in $\Omega_t$.  Finally, the diameter controls the circumradius, so we have
$$
\rho_+(t)\leq C(n)\sigma_0^{\frac{n(n+1-k)}{k}}r_{n+1}(\Omega_t).
$$
\endproof

\section{Upper bound of $E_k$}\label{sec:ubEk}
By \eqref{s3:radius}, the inner radius of $\Omega_t$ is bounded below by a positive constant $R_1$. This implies that there exists a ball of radius $R_1$ contained in $\Omega_t$ for all $t\in [0,T)$. The following lemma shows the existence of a ball with fixed centre enclosed by our flow hypersurfaces on a suitable fixed time interval.

\begin{lem}\label{lem-inball}
Let $M_t$ be a smooth convex solution of the flow \eqref{flow-VMCF} on $[0,T)$ with the global term $\phi(t)$ given by \eqref{s1:phi-3}. For any $t_0\in [0,T)$, let $B(p_0,\rho_0)$ be the inball of $\Omega_{t_0}$, where $\rho_0=\rho_-(\Omega_{t_0})$. Then
\begin{equation}\label{lem3-1-eqn1}
  B(p_0,\rho_0/2)\subset \Omega_t,\quad t\in [t_0, \min\{T,t_0+\tau\})
\end{equation}
for some $\tau$ depending only on $n,\alpha,\Omega_0$.
\end{lem}
\proof
Without loss of generality, we assume that $p_0$ is the origin. Then $\langle X,\nu\rangle\geq 0$ as long as $0\in M_t$, since $M_t$ is convex for each $t\in[0,T)$. Then
\begin{equation*}
  \frac{\partial}{\partial t}|X|^2=2(\phi-E_k^{\alpha/k})\langle X,\nu\rangle\geq~-2E_k^{\alpha/k}\langle X,\nu\rangle.
\end{equation*}
Denote $r(t)=\min_{M_t}|X|$. At the minimum point, we have $\langle X,\nu\rangle=|X|=r$ and the principal curvature $\kappa_i\leq r^{-1}$. Then
\begin{equation*}
  \frac d{dt}r(t)\geq -r^{-\alpha}
\end{equation*}
which implies that
\begin{equation*}
  r(t)\geq \left(\rho_0^{1+\alpha}-(\alpha+1)(t-t_0)\right)^{\frac 1{1+\alpha}}\geq \rho_0/2
\end{equation*}
provided that $t-t_0\leq (1+\alpha)^{-1}(1-2^{-\alpha-1})\rho_0^{\alpha+1}$. Let $\tau=(1+\alpha)^{-1}(1-2^{-\alpha-1})R_1^{\alpha+1}$, which depends only on $n,\alpha,\Omega_0$. Then $B(p_0,\rho_0/2)\subset \Omega_t$ for any $t\in [t_0, \min\{T,t_0+\tau\})$.
\endproof

Now we can use the technique that was first introduced by Tso \cite{Tso85} to prove the upper bound of $E_k$ along the flow \eqref{flow-VMCF}.

\begin{thm}\label{Ek-ub}
Let $M_t$ be a smooth convex solution of the flow \eqref{flow-VMCF} on $[0,T)$ with the global term $\phi(t)$ given by \eqref{s1:phi-3}. Then we have $\max_{M_t}E_k\leq C$ for any $t\in [0,T)$, where $C$ depends on $n,k,\alpha,M_0$ but not on $T$.
\end{thm}
\proof
For any given $t_0\in [0,T)$, using \eqref{lem3-1-eqn1} and the convexity of $M_t$, the support function $u=\langle X-p_0,\nu\rangle$ satisfies
\begin{equation}\label{s4:1}
  u-c\geq c>0,\quad \forall~t\in [t_0,\min\{T,t_0+\tau\}),
\end{equation}
where $c=\rho_0/4$. Define the function
\begin{equation*}%\label{def-W}
  W=\frac {E_k^{\alpha/k}}{u-c},
\end{equation*}
which is well-defined for all $t\in [t_0,\min\{T,t_0+\tau\})$. Combining \eqref{flow-VMCF} and \eqref{evl-nu} gives the evolution equation of $u(x,t)$ along the flow \eqref{flow-VMCF}
\begin{align}\label{evl-u}
   \frac{\partial}{\partial t}u =& \frac{\alpha}kE_{k}^{\frac{\alpha}k-1}\frac{\partial E_k}{\partial h_{ij}}\nabla_j\nabla_iu+\phi(t)-(1+\alpha)E_k^{\alpha/k}\nonumber\\
    &\quad +\frac{\alpha}kE_{k}^{\frac{\alpha}k-1} (nE_1E_k-(n-k)E_{k+1})u,
\end{align}
where $\nabla$ is the Levi-Civita connection on $M_t$ with respect to the induced metric. By \eqref{evl-El} and \eqref{evl-u}, we can compute the evolution equation of the function $W$
\begin{align*}%\label{evl-W-1}
   \frac{\partial}{\partial t}W= & \frac{\alpha}kE_{k}^{\frac{\alpha}k-1}\frac{\partial E_k}{\partial h_{ij}}\left(\nabla_j\nabla_iW+\frac 2{u-c}\nabla_iu\nabla_jW\right)\nonumber \\
  &\quad-\phi W\left( \frac{\alpha}k(nE_1-(n-k)\frac{E_{k+1}}{E_k})+\frac{1}{u-c}\right)\nonumber\\
  &\quad +(\alpha+1)W^2-\frac{\alpha c}kW^2(nE_1-(n-k)\frac{E_{k+1}}{E_k}).
\end{align*}
The Newton-MacLaurin inequality and \eqref{s4:1} implies that
\begin{equation*}
  nE_1-(n-k)\frac{E_{k+1}}{E_k}\geq kE_1,\quad\mathrm{ and }~W\leq E_1^{\alpha}/c.
\end{equation*}
Then we obtain the following estimate
\begin{align}\label{evl-W-2}
   \frac{\partial}{\partial t}W\leq & \frac{\alpha}kE_{k}^{\frac{\alpha}k-1}\frac{\partial E_k}{\partial h_{ij}}\left(\nabla_j\nabla_iW+\frac 2{u-c}\nabla_iu\nabla_jW\right)\nonumber \\
  &\quad +W^2\left(\alpha+1-\alpha c^{1+\frac 1{\alpha}}W^{1/{\alpha}}\right)
\end{align}
holds on $[t_0,\min\{T,t_0+\tau\})$. Let $\tilde{W}(t)=\sup_{M_t}W(\cdot,t)$. Then \eqref{evl-W-2} implies that
\begin{equation*}%\label{evl-W-3}
  \frac d{dt}\tilde{W}(t)\leq \tilde{W}^2\left(\alpha+1-\alpha c^{1+\frac 1{\alpha}}\tilde{W}^{1/{\alpha}}\right)
\end{equation*}
from which it follows from the maximum principle that
\begin{equation}\label{evl-W-4}
  \tilde{W}(t)\leq \max\left\{ \left(\frac {2(1+\alpha)}{\alpha}\right)^{\alpha}c^{-(\alpha+1)}, \left(\frac 2{1+\alpha}\right)^{\frac{\alpha}{1+\alpha}}c^{-1}t^{-\frac{\alpha}{1+\alpha}}\right\}.
\end{equation}
Then the upper bound on $E_k$ follows from \eqref{evl-W-4} and the facts that $c=\rho_0/4\geq R_1/4$ and $u-c\leq 2R_2$.
\endproof

Using the upper bound of $E_k$, we can prove the following estimate on the global term $\phi(t)$.
\begin{cor}\label{cor-phi-bd}
Let $M_t$ be a smooth convex solution of the flow \eqref{flow-VMCF} on $[0,T)$ with the global term $\phi(t)$ given by \eqref{s1:phi-3}. Then for any $p>0$ we have
\begin{equation}\label{cor-3-1-eqn1}
  0<C_1\leq \frac 1{|M_t|}\int_{M_t}E_k^pd\mu_t\leq C_2
\end{equation}
on $[0,T)$, where the constants $C_1,C_2$ depend only on $n,k,\alpha,p,M_0$. In particular,
\begin{equation}\label{cor-3-1-eqn2}
  C_1\leq~\phi(t)\leq C_2
\end{equation}
on $[0,T)$.
\end{cor}
\proof
The upper bound in \eqref{cor-3-1-eqn1} follows from the upper bound on $E_k$. For the lower bound in \eqref{cor-3-1-eqn1}, if $0<p\leq 1$, by the Alexandrov-Fenchel inequality \eqref{AF-k=n+1} and the upper bound on $E_k$,
 \begin{align*}
   |\Omega_t|^{\frac{n-k}{n+1}}\leq \omega_n^{-\frac{k+1}{n+1}}\int_{M_t}E_kd\mu_t\leq& \omega_n^{-\frac{k+1}{n+1}}\sup_{M_t}E_k^{1-p}\int_{M_t}E_k^pd\mu_t\\
   \leq & C(n,k,\alpha,p,M_0)\int_{M_t}E_k^pd\mu_t
 \end{align*}
 Then the lower bound in \eqref{cor-3-1-eqn1} follows from the above inequality and the upper bound on $|M_t|$. If $p>1$, the lower bound follows similarly by using the following inequality
  \begin{align*}
   |\Omega_t|^{\frac{n-k}{n+1}}\leq ~\omega_n^{-\frac{k+1}{n+1}}\int_{M_t}E_kd\mu_t\leq& ~ \omega_n^{-\frac{k+1}{n+1}}\left(\int_{M_t}E_k^pd\mu_t\right)^{1/p}|M_t|^{1-\frac 1p}.
 \end{align*}
\endproof

\section{Long-time existence}\label{sec:LTE}

Let $[0,T_{\max})$ be the maximum interval such that the solution of the flow \eqref{flow-VMCF} exists.
\begin{lem}\label{s5:lem1}
If $M_0$ is strictly convex, then $M_t$ is strictly convex for all $t\in[0,T_{\max})$.
\end{lem}
\proof
We consider the inverse second fundamental form $\tau_{ij}$, which evolves by \eqref{flow-tau2}.  At a maximum point and maximum eigenvector of $\tau_{ij}$, the first term on the right of \eqref{flow-tau2} is non-positive.  Also,
since $F=E_k^{1/k}$ is inverse-concave, i.e., $F_*$ is concave,  the second and third terms on the right of \eqref{flow-tau2}, involving $\bar\nabla\tau$, are non-positive.  The remaining terms applied to the maximum eigenvector of $\tau$ can be written as $-F_*^{-\alpha}-\alpha F_*^{-(1+\alpha)}\sum_i\dot F_*^i\left(\tau_{\max}-\tau_i\right)+\phi(t)\leq \phi(t)$.
 Moreover, $\phi(t)\leq C_2$ by \eqref{cor-3-1-eqn2}. It follows that the maximum eigenvalue of $\tau_{ij}$ is bounded by the maximum at $t=0$ plus $C_2t$, and in particular is bounded on any finite time interval.  The principal curvatures are the reciprocals of the eigenvalues of $\tau$, and so have a positive lower bound on any finite time interval.
 \endproof

\begin{thm}\label{s5:thm-LTE}
Let $M_0$ be a closed and strictly convex hypersurface in $\mathbb{R}^{n+1}$ and $M_t$ be the smooth solution of the flow \eqref{flow-VMCF} with the global term $\phi(t)$ given by \eqref{s1:phi-3}. Then $M_t$ is strictly convex and exists for all time $t\in [0,\infty)$.
\end{thm}
\proof
As shown in Lemma \ref{s5:lem1}, $M_t$ is strictly convex for all $t<T_{\max}$ and the principal curvatures are bounded below by a positive constant $\varepsilon$ (which may depend on $T_{\max}$). Theorem \ref{Ek-ub} gives us a uniform upper bound on $E_k$.   By the definition \eqref{eq:defEk} we have (ordering the principal curvatures in increasing order)
\begin{equation*}
  E_k\geq {n\choose k}^{-1}\kappa_{n-k+1}\cdots\kappa_n~\geq ~{n\choose k}^{-1}\varepsilon^{k-1}\kappa_n.
\end{equation*}
Thus the principal curvatures are uniformly bounded from above on $[0,T_{\max})$. Then we can argue as \cite[\S 6]{Cab-Sin2010} using the procedure in \cite{And2004,McC2005} to derive estimates on all higher derivatives of the curvature on $[0,T_{\max})$, and a standard continuation argument then shows that $T_{\max}=+\infty$.
\endproof

\section{Hausdorff Convergence}\label{sec:Haus}
In this section, we will prove that the flow \eqref{flow-VMCF} converges to a sphere in Hausdorff sense as $t\to\infty$.
Denote
\begin{equation*}
  \bar{E}_k=~\frac 1{|M_t|}\int_{M_t}E_kd\mu_t=~\frac{V_{n-k}(\Omega_t)}{V_n(\Omega_t)}.
\end{equation*}

\begin{lem}
Let $M_0$ be a closed and strictly convex hypersurface in $\mathbb{R}^{n+1}$ and $M_t$ be a smooth solution of the flow \eqref{flow-VMCF} with global term $\phi(t)$ given by \eqref{s1:phi-3}. Then there exists a sequence of times $t_i\to\infty$ such that
\begin{align}\label{5-3}
 \int_{M_{t_i}}\left(E_k-\bar{E}_k\right)^2d\mu_{t_i}~\to&~0, \quad\mathrm{ as}~i\to\infty
\end{align}
and
\begin{equation}\label{AF-0}
  V_{n+1-k}(\Omega_{t_i})-V_{n-k}(\Omega_{t_i})\frac{V_{n+1}(\Omega_{t_i})}{V_{n}(\Omega_{t_i})}~\to~0, \quad\mathrm{ as}~i\to\infty.
\end{equation}
\end{lem}
\proof
Proposition \ref{prop-monot} says that the isoperimetric ratio ${\mathcal I}_{n+1-k}(\Omega_t)$ is monotone non-increasing along the flow \eqref{flow-VMCF} with $\phi(t)$ given by \eqref{s1:phi-3}. Since
${\mathcal I}_{n+1-k}(\Omega_t)\geq 1$
for all $t$, there exists a sequence of times $t_i\to\infty$ such that
\begin{equation*}%\label{s6:lim1}
  \frac d{dt}\bigg|_{t_i}{\mathcal I}_{n+1-k}(\Omega_t)~\to~0,\quad\mathrm{ as}~i\to\infty.
\end{equation*}
Then from the proof of Proposition \ref{prop-monot}, the Jensen's inequality \eqref{s3:Jensen} approaches equality as $t_i\to\infty$, or equivalently
\begin{align}\label{5-1}
\int_{M_{t_i}}E_k^{\alpha/k}\left(E_k-\bar{E}_k\right)d\mu_{t_i}=\int_{M_{t_i}}\left(E_k^{\alpha/k}-\bar{E}_k^{\alpha/k}\right)\left(E_k-\bar{E}_k\right)d\mu_{t_i}~\to~0
\end{align}
as $i\to\infty$.
Now we observe that for $\alpha\geq k$, the convexity of $z\mapsto z^{\alpha/k}$ implies that
$$
\left(E_k^{\alpha/k}-\bar E_k^{\alpha/k}\right)\left(E_k-\bar E_k\right) \geq \frac{\alpha}{k}\bar E_k^{\alpha/k-1}\left(E_k-\bar E_k\right)^2.
$$
On the other hand, if $0<\alpha<k$ then
\begin{align*}
(E_k^{\alpha/k}-\bar E_k^{\alpha/k})(E_k-\bar E_k) &= \frac{\alpha}{k}\left(\int_0^1((1-s)E_k+s\bar E_k)^{\alpha/k-1}\,ds\right)(E_k-\bar E_k)^2\\
&\geq \frac{\alpha}{k}\left(\sup_{M_t}E_k\right)^{\alpha/k-1}\left(E_k-\bar E_k\right)^2.
\end{align*}
Since $E_k$ is bounded above by Theorem \ref{Ek-ub} and $\bar E_k$ is bounded below by Corollary \ref{cor-phi-bd}, we conclude in either case that
\begin{equation}\label{5-2}
  \left(E_k^{\alpha/k}-\bar{E}_k^{\alpha/k}\right)\left(E_k-\bar{E}_k\right)~\geq~C\left(E_k-\bar{E}_k\right)^2,
\end{equation}
so that \eqref{5-3} follows immediately from \eqref{5-1}.

For each time $t_i$, let $p_{t_i}$ be the center of the inball $B_{p_{t_i}}(\rho_-(\Omega_{t_i}))$ of $M_{t_i}$. By the estimate \eqref{s3:radius}, the support function $u_{t_i}=\langle X-p_{t_i},\nu\rangle$ of $M_{t_i}$ with respect to $p_{t_i}$ satisfies $|u_{t_i}|\leq 2R_2$ for all $t_i$. The H\"{o}lder inequality and \eqref{5-3} then imply that
 \begin{equation*}
   \int_{M_{t_i}}\left(E_k-\bar{E}_k\right)u_{t_i}~d\mu_{t_i}~\to~0,\quad \mathrm{as}~ i\to\infty.
 \end{equation*}
The estimate \eqref{AF-0} follows since the Minkowski formula gives that
\begin{align*}%\label{5-4}
 V_{n+1-k}(\Omega_{t_i})-V_{n-k}(\Omega_{t_i})\frac{V_{n+1}(\Omega_{t_i})}{V_{n}(\Omega_{t_i})}=&~\int_{M_{t_i}}E_{k-1}d\mu_{t_i}-\frac {(n+1)|\Omega_{t_i}|}{|M_{t_i}|}\int_{M_{t_i}}E_kd\mu_{t_i}\\
 =&~\int_{M_{t_i}}\left(E_k-\bar{E}_k\right)u_{t_i}~d\mu_{t_i}.
\end{align*}
\endproof
\begin{rem}
Note that the Alexandrov-Fenchel inequality \eqref{AF-1} implies that
\begin{equation*}
  V_{n+1-k}(\Omega)-V_{n-k}(\Omega)\frac{V_{n+1}(\Omega)}{V_{n}(\Omega)}~\geq ~0
\end{equation*}
for any convex body $\Omega$.
\end{rem}

\begin{lem}\label{s6:lem1}
Let $M_0$ be a closed and strictly convex hypersurface in $\mathbb{R}^{n+1}$ and $M_t$ be the smooth solution of the flow \eqref{flow-VMCF} with global term $\phi(t)$ given by \eqref{s1:phi-3}. Then there exists a sequence of times $t_i\to\infty$ such that $M_{t_i}$ converges to a round sphere $S^n(\bar{r})$ in Hausdorff sense as $t_i\to\infty$, where the radius $\bar{r}$ is determined by $\Omega_0$.
\end{lem}
\proof
(1). We first give an argument for the case $k=n$:  In this case the stability result of Lemma \ref{s2:stab-1} can be applied to prove the Hausdorff convergence.  We note that
\begin{align*}
V_1V_n-V_0V_{n+1}  &= \frac{V_1^2V_n^2-V_0^2V_{n+1}^2}{V_1V_n+V_0V_{n+1}}\\
&= \frac{V_n^2(V_1^2-V_0V_2)+V_0(V_2V_n^2-V_0V_{n+1}^2)}{V_1V_n+V_0V_{n+1}}\\
&\geq \frac{V_n^2}{V_1V_n+V_0V_{n+1}}\left(V_1^2-V_0V_2\right).
\end{align*}
It follows from \eqref{AF-0} that $V_1^2(\Omega_{t_i})-V_0V_2(\Omega_{t_i})\to 0$ as $i\to\infty$.  By the stability inequality \eqref{AF-2} there is a sequence of balls $B_i$ such that
$$
d_{\mathcal H}(\Omega_{t_i},B_i)\to 0.
$$
Here the diameter of $B_i$ equals $r_1(\Omega_{t_i})$ (which is bounded above and below by \eqref{s3:radius}, and the centre of $B_i$ is the Steiner point of $\Omega_{t_i}$, which remains in a bounded region by the Alexandrov reflection argument of \cite{Chow-Gul96}.  Therefore we can pass to a further subsequence along which $B_i$ converges in Hausdorff distance to a fixed ball $B$, and then we have
$$
d_{\mathcal H}(\Omega_{t_i},B)\to 0.
$$
We remark that this argument also applies for $k<n$ if Conjecture 7.6.13 in \cite{Schn} is true, since we have $E_k$ bounded above.  In particular, in the case where $\Omega_0$ is antipodally symmetric, then \cite{Schn}*{Theorem 7.6.20} provides the required statement, so this argument applies to prove Hausdorff convergence to a ball in this case for any $k$.

(2).  Next we provide a different argument which applies for all $k$, using the curvature measures ${\mathcal C}_m$ introduced in Section \ref{sec:2-4}:  By the estimate on the outer-radius of $\Omega_t$ in \eqref{s3:radius}, the Blaschke selection theorem (see Theorem 1.8.7 of \cite{Schn}) implies that there exists a subsequence of time $t_i$ and a convex body $\hat{\Omega}$ such that $\Omega_{t_i}$ converges to $\hat{\Omega}$ in Hausdorff sense as $t_i\to\infty$. As each $\Omega_{t_i}$ has inner radius $\rho_-(\Omega_{t_i})\geq R_1$, the limit convex body $\hat{\Omega}$ has positive inner radius. Without loss of generality, we may assume that the sequence $t_i$ is the same sequence such that \eqref{5-3} holds.

We will show that ${\mathcal C}_{n-k}(\hat\Omega,.) = c\,\mathcal C_{n}(\hat\Omega,.)$ where $c=\frac{V_{n-k}(\hat\Omega)}{V_n(\hat\Omega)}$.

The weak continuity of ${\mathcal C}_m$ is equivalent to the statement that $\int f d{\mathcal C}_m(\Omega_i)$ converges to $\int f d{\mathcal C}_m(\Omega)$ whenever $f$ is a bounded continuous function on ${\mathbb R}^{n+1}$ and $\Omega_i$ is a sequence of convex sets converging to $\Omega$ in Hausdorff distance.  In particular we have that $\int f d{\mathcal C}_{n-k}(\Omega_{t_i})$ converges to $\int f d{\mathcal C}_{n-k}(\hat\Omega)$, and $\int f d{\mathcal C}_{n}(\Omega_{t_i})$ converges to $\int f d{\mathcal C}_{n}(\hat\Omega)$, as $i\to\infty$.  Since $\Omega_{t_i}$ is smooth and uniformly convex, we have for any bounded continuous $f$
\begin{align*}
\left|\int f d{\mathcal C}_{n-k}(\Omega_{t_i}) -c \int f d{\mathcal C}_n(\Omega_{t_i})\right|
&= \left|\int_{M_{t_i}} f E_k d{\mathcal H}^n-\int_{M_{t_i}} f c d{\mathcal H}^n\right|\\
&\leq \sup |f| \int_{M_{t_i}}|E_k-c| d{\mathcal H}^n\\
&\leq \sup|f|\int_{M_{t_i}}|E_k-\bar E_k|d{\mathcal H}^n\\
&\quad\null+\sup|f|V_n(\Omega_{t_i})\left(\frac{V_{n-k}(\Omega_{t_i})}{V_n(\Omega_{t_i})}-\frac{V_{n-k}(\hat\Omega)}{V_n(\hat\Omega)}\right).
\end{align*}
The left-hand side converges to $\left|\int f d{\mathcal C}_{n-k}(\hat\Omega)-c\int f d{\mathcal C}_n(\hat\Omega)\right|$ by the weak continuity of the curvature measures, while the first term on the right-hand side converges to zero by \eqref{5-3}, and the second does also by Lemma \ref{s2:lem2}.  It follows that $\int fd{\mathcal C}_{n-k}(\hat\Omega) = c\int fd\mathcal C_n(\hat\Omega)$ for all bounded continuous functions $f$, and therefore that ${\mathcal C}_{n-k}(\hat\Omega,.) = c\,{\mathcal C}_n(\hat\Omega,.)$ as claimed.

By Theorem \ref{s2:thm-meas}, $\hat\Omega$ is a ball.
This completes the proof of Lemma \ref{s6:lem1}.

\endproof

Now we show that the whole family $M_t$ converges to a sphere as $t\to\infty$.

\begin{thm}\label{s6:thm-hd}
Let $M_0$ be a closed and strictly convex hypersurface in $\mathbb{R}^{n+1}$. Then the solution $M_{t}$ of \eqref{flow-VMCF} converges to a round sphere $S^n_{\bar{r}}(p)$ in Hausdorff distance as $t\to\infty$.
\end{thm}

\proof
By Proposition \ref{prop-monot}, $V_{n+1}(\Omega_t)$ is non-decreasing and $V_{n+1-k}(\Omega_t)$ is non-increasing.  It follows also that $V_{n+1}(\Omega_t)$ is bounded above, since $r_{n+1}(\Omega_t)\leq r_{n+1-k}(\Omega_t)\leq r_{n+1-k}(\Omega_0)$ by the Alexandrov-Fenchel inequality \eqref{s3:Il-2} and the monotonicity of $V_{n+1-k}(\Omega_t)$.    Similarly $V_{n+1-k}(\Omega_t)$ is bounded from below.  It follows that both converge as $t\to\infty$, and by Lemma \ref{s6:lem1} we have that ${\mathcal I}_{n+1-k}(\Omega_{t_i})\to 1$ as $i\to\infty$ and hence ${\mathcal I}_{n+1-k}(\Omega_t)\to 1$ as $t\to\infty$.  It follows that there exists $\bar r>0$ such that $r_{n+1}(\Omega_t)\to \bar r$ and $r_{n+1-k}(\Omega_t)\to \bar r$ as $t\to\infty$.  It follows from the stability estimate (7.124) in \cite{Schn} that $d_{\mathcal H}(\Omega_{t},B_{\bar r}(p(t)))\to 0$, where $p(t)$ is the Steiner point of $\Omega_t$.

To complete the argument we control the Steiner point $p(t)$ using an Alexandrov reflection argument:    For each $z\in S^n$ and $\lambda\in\RR$ we define a reflection $R_{z,\lambda}$ by $R_{z,\lambda}(x) = x-2(x\cdot z-\lambda)z$, which reflects in the hyperplane $\{x\cdot z=\lambda\}$.

We define $S^+(z,t) = \left\{\lambda\in\RR:\ R_{z,\lambda}(\Omega_t\cap\{x\cdot z>\lambda\})\subset\Omega\right\}$ and $\lambda^+(z,t) = \inf S^+(z,t)$.  The Alexandrov reflection argument \cite{Chow-Gul96} implies that $\lambda^+(z,t)$ is non-increasing in $t$ for each $z$.  Note also that $\lambda^-(z):=-\lambda^+(-z)$ satisfies $\lambda^-(z)\leq \lambda_+(z)$ and $\lambda^-(z,t)$ is non-decreasing in $t$ for each $z$.

\begin{lem}
If $d_{\mathcal H}(\Omega(t),B_{\bar r}(p(t)))<\varepsilon$, then $\lambda^+(z,t)\leq p(t)\cdot z+2\sqrt{\varepsilon \bar r}$.
\end{lem}

\begin{proof}
We fix $t$, and write $\Omega=\Omega_t$ and $p=p(t)$, so that by assumption $d_{\mathcal H}(\Omega, B_{\bar r}(p))<\varepsilon$.  Fix orthogonal unit vectors $z$ and $e$, and let $P=p+\text{\rm span}\{z,e\}\subset \RR^{n+1}$.   Then
$\Omega_P:=\Omega\cap P$ is a convex body in the plane $P$ which contains $\bar B_{\bar r-\varepsilon}(p)$ and is contained in $B_{\bar r+\varepsilon}(p)$.   For each $s$ we can write
$$
\Omega_P\cap\{x:\ (x-p)\cdot z=s\} = p+s z + (-u_P^-(s),u_P^+(s))e
$$
where $u_P^\pm$ are concave functions with
$$
\sqrt{(\bar r-\varepsilon)^2-s^2}\leq u_P^\pm(s)<\sqrt{(\bar r+\varepsilon)^2-s^2}
$$
for each $0\leq s<\bar r-\varepsilon$.  In particular we have $\bar r-\varepsilon< u_P^\pm(0)$, and
$$
u_P^\pm(2\sqrt{\bar r\varepsilon})<\sqrt{(\bar r+\varepsilon)^2-4\bar r\varepsilon} = \bar r-\varepsilon.
$$
It follows by concavity of $u_P^\pm$ that $u_P^{\pm}(s)>u_P(2\sqrt{\bar r\varepsilon})$ for $0\leq s\leq 2\sqrt{\bar r\varepsilon}$, and $u_P^\pm(s)<u_P^\pm(2\sqrt{\bar r\varepsilon})$ for $s>2\sqrt{\bar r\varepsilon}$.  It follows that $x=p+s z+ ye\in \Omega_P$ with $s>2\sqrt{\bar r\varepsilon}$ implies that
\begin{align*}
R_{z,p\cdot z+2\sqrt{\bar r\varepsilon}}(x) &= p+(4\sqrt{\bar r\varepsilon}-s)z+ye\\
&\in p+(4\sqrt{\bar r\varepsilon}-s)z+(-u_P^-(s),u_P^+(s))e\\
&\subset p+(4\sqrt{\bar r\varepsilon}-s)z+(-u_P^-(2\sqrt{\bar r\varepsilon}),u_P^+(2\sqrt{\bar r\varepsilon}))e\\
&\subset p+(4\sqrt{\bar r\varepsilon}-s)z+(-u_P^-(4\sqrt{\bar r\varepsilon}-s),u_P^+(4\sqrt{\bar r\varepsilon}-s))e\\
&\subset \Omega_P\subset\Omega
\end{align*}
for each $s\in[0,4\sqrt{\bar r\varepsilon}]$.  Furthermore for $x\in\Omega$ with $(x-p)\cdot z>4\sqrt{\bar r\varepsilon}$ we have
$$
R_{z,p\cdot z+2\sqrt{\bar r\varepsilon}}(x) \in B_{\bar r-\varepsilon}(p)\subset\Omega.
$$
Since $e$ is arbitrary, we have that $R_{z,p\cdot z+2\sqrt{\bar r\varepsilon}}(\Omega\cap\{(x-p)\cdot z>2\sqrt{\bar r\varepsilon}\})\subset\Omega$, and so $p\cdot z+2\sqrt{\bar r\varepsilon}\in S^+(z,t)$ and $\lambda^+(z,t)\leq p\cdot z+2\sqrt{\bar r\varepsilon}$ as claimed.
\end{proof}

It follows from the monotonicity of $\lambda_+(z,t)$ that if $d_{\bar H}(\Omega(t),B_{\bar r}(p(t)))<\varepsilon$ then for $t'>t$ we have
$$
p(t)\cdot z-2\sqrt{\bar r\varepsilon}\leq \lambda_-(z,t)\leq \lambda_-(z,t')\leq \lambda_+(z,t')\leq \lambda_+(z,t)\leq p(t)\cdot z+2\sqrt{\bar r\varepsilon}
$$
so that
$$
\left|p(t)\cdot z-\lambda_\pm(z,t')\right|<2\sqrt{\bar r\varepsilon}.
$$
It follows that $\lambda_\pm(z,t')$ is Cauchy, hence convergent, as $t'\to\infty$, and that $p(t)\cdot z$ converges to the same limit.  Since $z$ is arbitrary, this proves that $p(t)$ converges to a point $p\in\RR^{n+1}$, and we conclude that
$d_{\mathcal H}(\Omega(t),p)\to 0$ as $t\to\infty$.
\endproof

\section{Smooth convergence}\label{sec:Smooth}
We proved in the previous section the Hausdorff convergence of the solution $M_{t}$ of \eqref{flow-VMCF} to a round sphere $S^n(\bar{r})$ as $t\to\infty$. In this section, we prove that the convergence is in the $C^\infty$ sense.  Firstly, we prove the following lower bound for $E_k$ along the flow \eqref{flow-VMCF}, by adapting an idea of Smoczyk \cite[Proposition 4]{Smo98} (see also \cite{Andrews-McCoy-Zheng,McC2017}).

\begin{prop}\label{s7:prop1}
Let $M_0$ be a closed and strictly convex hypersurface in $\mathbb{R}^{n+1}$ and $M_t$, $t\in [0,\infty),$ be the smooth solution of the flow \eqref{flow-VMCF}.  Then there exists a positive constant $C=C(n,\alpha,\Omega_0)$ such that $E_k\geq C$ on $M_t$ for all $t\geq 0$.
\end{prop}
\proof
The convexity estimate of Lemma \ref{s5:lem1} implies a lower bound on $E_k$ on any finite time interval, so we need only control $E_k$ from below for sufficiently large times.
As shown in Corollary \ref{cor-phi-bd}, $\phi(t)$ is bounded from above and below by positive constants:
\begin{equation*}
  0<\phi_-\leq \phi(t)\leq~\phi_+,
\end{equation*}
where $\phi_-,\phi_+$ are positive uniform constants depending only on $n,k,\alpha$ and $M_0$. Since $\Omega_t$ converges to the ball $B(\bar{r})$ in Hausdorff sense as $t\to\infty$, for any small $\epsilon>0$, there exists a large time $t_0>0$ such that for all $t\geq t_0$, the outer radius $\rho_+(t)$ and inner radius $\rho_-(t)$ of $\Omega_t$ satisfies
\begin{equation}\label{s7:rho-+}
  (1-\epsilon)\bar{r}~\leq~ \rho_-(t)\leq~\rho_+(t)\leq~(1+\epsilon)\bar{r}.
\end{equation}
In the following, we only consider times $t\geq t_0$. Let $t_1>t_0$ and set
\begin{equation*}
  \varphi_{t_1}(t)=\int_{t_1}^t\phi(s)ds.
\end{equation*}
Define $f(z,t)$ by
\begin{equation}\label{s7:f-def}
  f(z,t)~=~-(1+\alpha)(t-t_1)\Psi(z,t)-(1+\alpha)\varphi_{t_1}(t)+s(z,t)
\end{equation}
for $t\geq t_1$. Combining \eqref{flow-gauss} and \eqref{flow-Psi} gives the the evolution equation of $f(z,t)$
\begin{align}\label{s7:dt-f1}
   \frac{\partial}{\partial t}f (z,t)= & ~\dot{\Psi}^{kl}(z,t)\bar{\nabla}_k\bar{\nabla}_lf(z,t)-\alpha\phi(t) \nonumber\\ &\quad+\biggl(s(z,t)-(1+\alpha)(t-t_1)(\Psi+\phi(t))\biggr)\dot{\Psi}^{kl}\bar{g}_{kl}.
\end{align}
We now show that the last term on the right hand side of \eqref{s7:dt-f1} is nonnegative for a short time after time $t_1$. Choose the origin to be the center of the inball of $\Omega_{t_1}$.
Then
\begin{equation*}
s(z,t_1)\geq~s(z_1,t_1) =\rho_-(t_1)\geq~(1-\epsilon)\bar{r},
\end{equation*}
where $z_1$ is the minimum point of $s(z,t_1)$. By Theorem \ref{Ek-ub}, $E_k$ is bounded above by a uniform constant. Therefore $\Psi=-E_k^{\alpha/k}$ is bounded below by a uniform negative constant $-c_+$. Then the evolution equation \eqref{flow-gauss} of $s(z,t)$ gives that
\begin{align*}
  s(z,t)\geq & ~s(z,t_1)+(\phi_--c_+)(t-t_1) \\
   \geq & (1-\epsilon)\bar{r}+(\phi_--c_+)(t-t_1).
\end{align*}
Let $t_2>t_1$ be the time such that
\begin{equation}\label{s7:t2-1}
  t_2-t_1=~\frac{(1-\epsilon)\bar{r}}{c_+-\phi_-+(1+\alpha)\phi_+},
\end{equation}
which is positive and independent of time $t_1$. Then for any $t\in [t_1,t_2]$, we have
\begin{align}\label{s7:s}
& s(z,t)-(1+\alpha)(t-t_1)(\Psi(z,t)+\phi(t)) \nonumber\\
 & \geq  -(1+\alpha)(t-t_1)\phi_++(1-\epsilon)\bar{r}+(\phi_--c_+)(t-t_1)\geq~0.
\end{align}
Since $\dot{\Psi}^{kl}\bar{g}_{kl}\geq 0$, combining \eqref{s7:s} with \eqref{s7:dt-f1} gives that
\begin{align}\label{s7:dt-f2}
   \frac{\partial}{\partial t}f (z,t)
   \geq&~\dot{\Psi}^{kl}(z,t)\bar{\nabla}_k\bar{\nabla}_lf(z,t)-\alpha\phi(t)
\end{align}
for all time $t\in [t_1,t_2]$. The maximum principle implies that
\begin{align}\label{s7:s2}
  f(z,t) \geq & \min_{M_t}f(z,t) \nonumber\\
   \geq & \min_{M_{t_1}}f(z,t_1)-\alpha\varphi_{t_1}(t)\nonumber\\
   =&s(z_1,t_1)-\alpha\varphi_{t_1}(t)
\end{align}
for all time $t\in [t_1,t_2]$. By the definition \eqref{s7:f-def} of $f(z,t)$, \eqref{s7:s2} implies that
 \begin{equation}\label{s7:Psi-lbd1}
   -(1+\alpha)\Psi(z,t)\geq~\frac{\varphi_{t_1}(t)-s(z,t)+s(z_1,t_1)}{t-t_1}
 \end{equation}
for all $t\in (t_1,t_2]$ with $t_2$ defined in \eqref{s7:t2-1}.

To estimate the lower bound of $-\Psi$, it remains to estimate the right hand side of \eqref{s7:Psi-lbd1}. Let $r(t)=\max_{M_t}|X|$. Recall that the origin is chosen to be the center of the inball of $\Omega_{t_1}$, then
\begin{equation}\label{s7:r+}
  r_+:=r(t_1)\leq (1+2\epsilon)\bar{r}.
\end{equation}
From the flow equation \eqref{flow-VMCF}, and noting that the principal curvature $\kappa_i\geq r(t)^{-1}$ at the maximal point of $|X|$, we deduce that
\begin{equation}\label{s7:rt-2}
   \frac d{dt}r(t) \leq~\phi(t)-r(t)^{-\alpha}\leq ~\phi_+,
 \end{equation}
 which implies that
 \begin{equation*}
   r(t)\leq~r_++\phi_+(t-t_1).
 \end{equation*}
 Then
 \begin{equation}\label{s7:rt-3}
   r(t)\leq 2r_+, \quad \forall~t\in [t_1,t_3],
 \end{equation}
where $t_3$ is given by $t_3-t_1= \phi_+^{-1}r_+$. Since
\begin{align*}
  \frac d{dt}\left(\rho_-(t_1)+\varphi_{t_1}(t)-r(t)\right) \geq &~r(t)^{-\alpha}\geq ~2^{-\alpha}r_+^{-\alpha}, \quad \forall~t\in [t_1,t_3],
\end{align*}
we have
\begin{align}\label{s7:rt-4}
 \rho_-(t_1)+\varphi_{t_1}(t)-r(t)\geq &~\rho_-(t_1)-r_++2^{-\alpha}r_+^{-\alpha}(t-t_1)\nonumber\\
 \geq &~ 2^{-1-\alpha}r_+^{-\alpha}(t-t_1)
\end{align}
provided that
\begin{equation*}
  t\geq~t_1+2^{1+\alpha}r_+^{\alpha}(r_+-\rho_-(t_1))~=:t_4.
\end{equation*}
By \eqref{s7:rho-+} and \eqref{s7:r+}, we know that
\begin{equation*}
  r_+-\rho_-(t_1)\leq ~3\epsilon \bar{r}
\end{equation*}
can be arbitrary small by assuming $t_0$ is sufficiently large. Then the waiting time $t_4-t_1$ can be made small such that $t_4<\min\{t_2,t_3\}$ by choosing $t_0$ sufficiently large. Combining \eqref{s7:Psi-lbd1} and \eqref{s7:rt-4} yields that
 \begin{equation}\label{s7:Psi-lbd3}
   -(1+\alpha)\Psi(z,t)\geq~\frac{\rho_-(t_1)+\varphi_{t_1}(t)-r(t)}{t-t_1}\geq~2^{-1-\alpha}r_+^{-\alpha}
 \end{equation}
 for all time $t\in[t_4,\min\{t_2,t_3\}]$. This gives the uniform positive lower bound on $E_k$.
\endproof

The lower speed bound allows us to obtain a uniform lower bound on principal curvatures:

\begin{prop}\label{prop:unifconv}
If $M_0$ is smooth and uniformly convex, then there exists $\varepsilon>0$ such that $\kappa_i(x,t)\geq \varepsilon$ for all $i$ and for all $(x,t)\in M\times[0,\infty)$.
\end{prop}

\begin{proof}
In the evolution equation \eqref{flow-tau2} for $\tau_{ij}$, all terms are non-positive except for the last one.  The lower bound on $E_k$ implies a lower bound on $F_*^{-1}$, so the last two terms are bounded by $-C_1\tau_{ij}+C_2\bar g_{ij}$.  This is negative if $\tau_{ij}$ is sufficiently large, so any sufficiently large upper bound on the eigenvalues of $\tau$ is preserved.  Equivalently, and sufficiently small lower bound on principal curvatures is preserved by the flow.
\end{proof}

The uniform upper bounds on $E_k$, together with the uniform lower bound on $\kappa_i$, implies a uniform upper bound on principal curvatures by the argument of Theorem \ref{s5:thm-LTE}.   It follows that
 the flow \eqref{flow-VMCF} is uniformly parabolic for all $t>0$.  Then an argument similar to that in the proof of Theorem \ref{s5:thm-LTE} can be applied to show that all derivatives of curvatures are uniformly bounded on $M_t$ for all $t>0$.  Since $M_t$ converges in Hausdorff distance to $S_{\bar r}(p)$, this implies that $M_t$ converges in $C^\infty$ to $S_{\bar r}(p)$.   By an argument similar to that in \cite{And1994-2}, we can further prove that $X(\cdot,t)$ converges exponentially fast to a smooth limiting embedding with image equal to the sphere $S_{\bar r}(p)$ as $t\to\infty$.

\section{Generalizations}\label{sec:8}
In this section, we discuss some generalisations of Theorem \ref{thm1-1}.

\subsection{Nonhomogeneous functions of $E_k$}

Recently Bertini and Sinestrari \cite{Bert-Sin} have considered constrained flows by rather general increasing functions of the mean curvature, preserving either the enclosed volume or the $n$-dimensional area of the evolving hypersurfaces.  Our methods allow a similar result for constrained flows by powers of the elementary symmetric functions $E_k$:

\begin{thm}
Let $\mu:\ {\mathbb R}_+\to{\mathbb R}_+$ be a smooth function with $\mu'(z)>0$ for $z>0$, and with $\lim_{z\to 0}\mu(z)=0$ and $\lim_{z\to\infty}\mu(z)=\infty$.  Suppose further that $\mu(\xi^{-1})$ is a convex function of $\xi$ (equivalently, $\mu''(z)+\frac{2}{z}\mu'(z)\geq 0$ for all $z>0$), and such that $\lim_{x\to\infty}\frac{\mu'(x)x^2}{\mu(x)}=\infty$, while $\frac{z\mu'(z)}{\mu(z)}=O(1)$ as $z\to 0$.  Then the flow
\begin{equation}\label{eq:NHF}
\frac{\partial X}{\partial t}(x,t) = \left(\phi(t)-\mu(E_k^\frac1k(x,t))\right)\nu(x,t)
\end{equation}
with $\phi(t)$ chosen to keep $G(r_{n+1-k}(\Omega_t),r_{n+1}(\Omega_t))$ fixed, has a smooth solution for any smooth, uniformly convex initial embedding $X_0$, which exists for all positive times and converges smoothly as $t\to\infty$ to a limiting embedding  $X_\infty$ with image equal to a sphere $S^n_{\bar r}(p)$ for some $p\in\RR^{n+1}$, where $G(\bar r,\bar r) = G(r_{n+1-k}(\Omega_0),r_{n+1}(\Omega_0))$.
\end{thm}

We briefly mention the key steps involved:  The requirement to preserve the constraint $G$ implies that the global term $\phi(t)$ is given by the following expression:
$$
\phi = \frac{G_a\omega_n^{\frac{1}{n+1-k}}V_{n+1-k}^{\frac{n-k}{n+1-k}}\int_M\mu + G_b\omega_n^{\frac{1}{n+1}}V_{n+1}^{\frac{n}{n+1}}\int_ME_k\mu}{G_a\omega_n^{\frac{1}{n+1-k}}V_{n+1-k}^{\frac{n-k}{n+1-k}}V_n + G_b\omega_n^{\frac{1}{n+1}}V_{n+1}^{\frac{n}{n+1}}V_{n-k}}.
$$
As in the homogeneous cases, the enclosed volume is non-decreasing and the mixed volume $V_{n-k}(\Omega_t)$ is non-increasing:  We have
\begin{align*}
\frac{d}{dt}V_{n+1}(\Omega_t) &=\frac{(n+1)G_b\omega_n^{\frac{1}{n+1}}V_{n+1}^{\frac{n}{n+1}}}{G_a\omega_n^{\frac{1}{n+1-k}}V_{n+1-k}^{\frac{n-k}{n+1-k}}V_n + G_b\omega_n^{\frac{1}{n+1}}V_{n+1}^{\frac{n}{n+1}}V_{n-k}}\left(|M_t|\int_{M_t}E_k\mu-\int_{M_t}E_k\,\int_{M_t}\mu\right).
\end{align*}
The bracket on the right is non-negative, since
\begin{align}
|M_t|\int_{M_t}E_k\mu-\int_{M_t}E_k\int_{M_t}\mu &= \int_{M_t}\int_{M_t}E_k(x)\mu(x)-E_k(x)\mu(y)\,d{\mathcal H}^n(x)d{\mathcal H}^n(y)\notag\\
&=\frac12\int_{M_t}\int_{M_t}E_k(x)\mu(x)+E_k(y)\mu(y)\notag\\
&\quad\null\phantom{\frac12\int_{M_t}\int_{M_t}}-E_k(x)\mu(y)-E_k(y)\mu(x)\,d{\mathcal H}^n(x)d{\mathcal H}^n(y)\notag\\
&=\frac12\int_{M_t}\int_{M_t}\left(E_k(x)-E_k(y)\right)\left(\mu(x)-\mu(y)\right)\,d{\mathcal H}^n(x)d{\mathcal H}^n(y)\label{eq:baineq}
\end{align}
which is non-negative since $\mu(E_k^{1/k})$ is an increasing function of $E_k$.   The argument to show that $V_{n+1-k}(\Omega_t)$ is non-increasing is similar, and it follows that the isoperimetric ratio
${\mathcal I}_{n+1-k}$ is strictly decreasing unless the hypersurface is a sphere.  This gives an upper bound on diameter and a lower bound on inradius.  As in \cite{Bert-Sin}, an upper bound on $E_k$ follows by an argument similar to that in Theorem \ref{Ek-ub}, provided that $\lim_{x\to\infty}\frac{\mu'(x)x^2}{\mu(x)}=\infty$.   The argument of Lemma \ref{s5:lem1} to preserve convexity of the evolving hypersurface also applies, provided we assume that $\mu(\xi^{-1})$ is a convex function of $\xi$ for $\xi>0$ (equivalently
$z\mu''(z)+2\mu'(z)\geq 0$ for all $z>0$).  The Hausdorff convergence argument of Lemma \ref{s6:lem1} applies without change since we can find a subsequence of times $t_i$ along which $\frac{d}{dt}{\mathcal I}_{n+1-k}$ approaches zero.  This implies that $|M_{t_i}|\int_{M_{t_i}}E_k\mu - \int_{M_{t_i}}E_k\int_{M_{t_i}}\mu$ approaches zero, which is sufficient to ensure that $\int_{M_{t_i}}|E_k-\bar E_k|\,d{\mathcal H}^n$ approaches zero.  The lower bound on $E_k$ given in Proposition \ref{s7:prop1} holds provided that we assume that $\frac{z\mu'(z)}{\mu(z)}$ is bounded as $z\to 0$, and this also suffices for the uniform lower bound on principal curvatures, following the argument of Proposition \ref{prop:unifconv}.  As before, the flow is then uniformly parabolic and concave in $\tau$, and higher regularity and smooth convergence follow.

\subsection{Volume non-decreasing flows}

We can more generally consider Equation \eqref{eq:NHF} with the global term $\phi(t)$ satisfying
\begin{equation}\label{s8:phi-0}
  \phi(t)\geq \bar\mu:=\frac 1{V_n(\Omega_t)}\int_{M_t}\mu(E_k^{1/k})d\mu_t.
\end{equation}
By \eqref{s3:Vn1}, this means the volume of $\Omega_t$ is monotone non-decreasing along the flow.  We can check that the isoperimetric ratio ${\mathcal I}_{n+1-k}(\Omega_t)$ is non-increasing along the flow:  We have
\begin{align*}
\frac{1}{(n+1)(n+1-k)}\frac{d}{dt}{\mathcal I}_{n+1-k} &= (\phi(t)-\bar\mu)\frac{V_{n+1-k}^{n}}{V_{n+1}^{n+2-k}V_0^k}\left(V_{n+1}V_{n-k}-V_{n+1-k}V_n\right)\\
&\quad\null + \frac{V_{n+1-k}^{n}}{V_{n+1}^{n+1-k}V_0^kV_n}\left(V_{n-k}\int_{M_t}\mu-V_n\int_{M_t}E_k\mu\right).
\end{align*}
The bracket on the first line is non-positive by the Alexandrov-Fenchel inequality, while the second is non-positive by the expression \eqref{eq:baineq}.

In the case where the volume $V_{n+1}$ remains bounded, the analysis is essentially the same as in the previous case.  Otherwise the volume tends to infinity, and the Alexandrov reflection argument implies that $d_{\mathcal H}(\Omega_t,B_{r_{n+1}(\Omega_t)}(p))$ remains bounded, for some fixed $p\in\RR^{n+1}$.  This implies in particular that rescaling about $p$ to fixed enclosed volume gives Hausdorff convergence to a ball.  However it seems likely that further assumptions are required to prove smooth convergence in this case.

\subsection{Anisotropic generalisations}

If $W$ is a smooth, uniformly convex body containing the origin in $\RR^{n+1}$, then there is an associated `relative differential geometry' defined by $W$:  On any smooth (oriented) hypersurface, $\nu_W:\ M\to W$ is the (smooth) map which takes each point $x\in M$ to the point in $W$ with the same oriented tangent plane.  The relative curvature ${\mathcal W}_W$ is the derivative of $\nu_W$, which is a linear map from $T_xM$ to $T_xM$ at each point.  The relative principal curvatures are the eigenvalues of ${\mathcal W}_W$.

In general we can consider flows of the form
$$
\frac{\partial}{\partial t}X(p,t) = (\phi(t)-F({\mathcal W}_W))\nu_W,
$$
where $F$ is a symmetric, monotone function of the relative principal curvatures.

A large part of the analysis we have given for the isotropic flows carries through directly for their anisotropic analogues:  If we replace the mixed volumes $V_j$ by their anisotropic analogues (defined by replacing the ball $B$ by the Wulff shape $W$), then the evolution equations for the mixed volumes are formally unchanged, and in particular choosing $F=\mu(E_k)$, where $\mu$ is an increasing function and $E_k$ is the $k$th elementary symmetric function of the relative principal curvatures, and choosing $\phi(t)$ to preserve a monotone function of $V_{n+1}$ and $V_{n+1-k}$, one can show that the enclosed volume is non-decreasing, while the relative mixed volume $V_{n+1-k}(\Omega_t)$ is non-increasing.

The only point of departure from our analysis for the isotropic case is that there is no known anisotropic analogue of Theorem \ref{s2:thm-meas}.  Thus the analogous result holds for $k=n$ (using the first of the two arguments in the proof of Lemma \ref{s6:lem1}), but the result for $k<n$ cannot so far be deduced.    The full result for anisotropic cases would follow from a natural conjecture concerning the anisotropic curvature measures, which we now define:

First we define a (non-symmetric) distance relative to $W$, by setting
$$
d_W(x,y) = \inf\{r>0:\ y\in x+rW\}.
$$
The relative distance of a point from a set $\Omega$ is then defined by
$$
d_W(\Omega,y) = \inf\{d_W(x,y):\ x\in\Omega\}.
$$
If $\Omega$ is a convex body, then the infimum in the last definition is attained at a single point which we denote by $p_W(\Omega,y)\in\partial\Omega$.  Finally, if $\Omega\subset\RR^{n+1}$ is an open bounded convex body, $\beta$ is an open set in $\RR^{n+1}$, and $\rho>0$, then we define
$$
A_\rho^W(\Omega,\beta) = \{x\in\RR^{n+1}:\ 0<d_W(\Omega,x)<\rho,\ p_W(\Omega,x)\in\beta\}.
$$
The anisotropic curvature measures are then defined by the expansion
$$
{\mathcal H}^{n+1}(A_\rho^W(\Omega,\beta)) = \frac{1}{n+1}\sum_{m=0}^n\rho^{n+1-m}{n+1\choose m}{\mathcal C}_m^W(\Omega,\beta).
$$

\begin{con}
If $\Omega$ is an open bounded convex body in $\RR^{n+1}$ with ${\mathcal C}_m^W(\Omega,.) = c\,{\mathcal C}_n^W(\Omega,.)$ for some $m\in\{0,\cdots,n-1\}$ and $c>0$, then $\Omega=\lambda W+p$ for some $\lambda>0$ and $p\in\RR^{n+1}$.
\end{con}

This conjecture would suffice to prove the analogue of our Theorem \ref{thm1-1} for the anisotropic flows corresponding to any smooth, uniformly convex Wulff shape $W$ containing the origin.

The case $k=1$ can be dealt with by different arguments:  Following the argument in \cite{Bert-Sin}, one can prove a lower bound on the speed $\mu(E_1)$ directly from the maximum principle (under assumptions on $\mu$ identical to those in \cite{Bert-Sin}).  Once the lower speed bound is obtained, the flow is uniformly parabolic and estimates on all higher derivatives can be deduced by standard arguments.  It follows that the solution converges smoothly to a limiting hypersurface which is smooth, uniformly convex, and has constant $E_1({\mathcal W}_W)$ (by the evolution equation for the isoperimetric ratio).  By a rigidity result for equality cases in the Alexandrov-Fenchel inequalities under the assumption of uniform convexity \cite{Schn}*{Theorem 7.6.8} the limit is a scaled translate of the Wulff shape $W$.

\subsection{Flows without divergence structure}

The maximum principle argument of Bertini and Sinestrari \cite{Bert-Sin} to obtain a lower bound on mean curvature under flows by functions of mean curvature cannot be extended to flows involving other elementary symmetric functions $E_k$, since the functions $E_k^{1/k}$ are not uniformly elliptic unless a curvature pinching estimate is known.  However, their argument can be usefully employed for a large class of other flows:

Consider flows of the form
$$
\frac{\partial X}{\partial t} = (\phi(t)-\mu(F({\mathcal W}))\nu,
$$
where $\mu$ has positive derivative, and $F$ is uniformly elliptic (so that $F(A)+\lambda\text{\rm Tr}(B)\leq F(A+B)\leq F(A)+\Lambda\text{\rm Tr}(B)$ for any positive definite matrices $A$ and $B$).  Then the argument of \cite{Bert-Sin} applies to produce a lower speed bound, and then the flow is uniformly parabolic so that the solution has all higher derivatives bounded, provided that $F$ is either concave or inverse-concave so that H\"older continuity of the second derivatives can be deduced.

For these class of flows it is no longer the case that an isoperimetric ratio improves, but we can proceed by considering the Alexandrov reflection argument outlined in the proof of Theorem \ref{s6:thm-hd}.  The higher derivative estimates allow us to produce a limiting solution of the flow as $t\to\infty$, and for this limit the monotone quantities $\lambda_{\pm}(z,t)$ arising from the Alexandrov reflection argument must be constant in time.  A strong maximum principle then implies that the hypersurface must have reflection symmetries in every direction $z$, and so is a sphere.

%-------------------------------------------------
\end{document}